\documentclass[12pt,letterpaper]{amsart}

\title[Noise Sensitivity and Learning Lower Bounds for Hierarchical Functions]{Noise Sensitivity and Learning Lower Bounds for Hierarchical Functions}

\usepackage[utf8]{inputenc}
\usepackage{amsmath, amssymb, amsthm}
\usepackage{amsfonts, amsxtra, tikz, hyperref, comment, url, wasysym, import, enumitem, breqn, dsfont, thmtools}
\usepackage{subcaption}
\usepackage[capitalise,noabbrev,nameinlink]{cleveref}
\usetikzlibrary{matrix}
\calclayout

\definecolor{lavender}{rgb}{0.5,0,1.0}
\hypersetup{colorlinks=true, citecolor=lavender, linkcolor=lavender,urlcolor=lavender}
\newcommand{\dfn}[1]{\textcolor{blue}{\emph{#1}}}

\usepackage{todonotes}

\newcommand{\SortNoop}[1]{}

\usepackage{doi}

\def\E{\mathbb{E}}

\def\P{\mathbb{P}}

\def\R{\mathbb{R}}

\def\Z{\mathbb{Z}}
\def\eps{\varepsilon}

\def\deg{\operatorname{deg}}
\def\0{\mathbf{0}}

\def\Var{\operatorname{Var}}

\def\Unif{\operatorname{Unif}}

\def\Corr{\operatorname{Corr}}
\def\Cov{\operatorname{Cov}}
\def\Lin{\operatorname{Lin}}
\def\maj{\operatorname{maj}}

\newcommand{\floor}[1]{\lfloor#1\rfloor}

\newcommand{\mf}[1]{\mathfrak{#1}}
\newcommand{\mc}[1]{\mathcal{#1}}
\newcommand{\norm}[1]{\left\lVert#1\right\rVert}
\newcommand{\abs}[1]{\left\lvert#1\right\rvert}
\newcommand{\set}[1]{\left\{#1\right\}}
\newcommand{\paren}[1]{\left(#1\right)}
\newcommand{\bracket}[1]{\left[#1\right]}
\newcommand{\1}[1]{1_{\left\{#1\right\}}}

\newtheorem{theorem}{Theorem}[section]
\newtheorem{lemma}[theorem]{Lemma}

\newtheorem{corollary}[theorem]{Corollary}

\theoremstyle{definition}
\newtheorem{definition}[theorem]{Definition}
\newtheorem{example}[theorem]{Example}

\crefname{problem}{Problem}{Problems}

\theoremstyle{remark}
\newtheorem{remark}[theorem]{Remark}

\counterwithin{equation}{section}
\begin{document}

\author[Rupert Li]{Rupert Li}
\address[]{Stanford University, Stanford, CA 94305, USA}
\email{rupertli@stanford.edu}

\author[Elchanan Mossel]{Elchanan Mossel}
\address[]{Massachusetts Institute of Technology, Cambridge, MA 02139, USA}
\email{elmos@mit.edu}

\begin{abstract}
    Recent works explore deep learning's success by examining functions or data with hierarchical structure.
    To study the learning complexity of functions with hierarchical structure, 
    we study the noise stability of functions with tree hierarchical structure on independent inputs.
    We show that if each function in the hierarchy is $\varepsilon$-far from linear,
    the noise stability is exponentially small in the depth of the hierarchy. 
    
    Our results have immediate applications for agnostic learning.
    In the Boolean setting using the results of
    Dachman-Soled, Feldman, Tan, Wan and Wimmer (2014), our results provide Statistical Query super-polynomial lower bounds for agnostically learning classes that are based on hierarchical functions.
    
    We also derive similar SQ lower bounds based on the indicators of crossing events in critical site percolation. These crossing events are not formally hierarchical as we define but still have some hierarchical features as studied in  mathematical physics.
    
    Using the results of Abbe, Bengio, Cornacchiam, Kleinberg, Lotfi, Raghu and Zhang (2022), our results imply sample complexity lower bounds for learning hierarchical functions with gradient descent on fully connected neural networks. 
    
    Finally in the Gaussian setting, using the results of Diakonikolas, Kane, Pittas and Zarifis (2021), our results provide super-polynomial lower bounds for agnostic SQ learning.
\end{abstract}

% \begin{keywords}
%     Noise sensitivity, hierarchical functions, deep learning lower bounds, Boolean functions, noise stability, multilinear functions, non-separable functions, Efron--Stein decomposition, percolation
% \end{keywords}

\maketitle

\section{Introduction}\label{sec:introduction}

Our goal is to bridge two lines of research in learning theory. The first line proposed that the success of neural networks can be explained in terms 
 of hierarchical structures in the data. 
 The second line of work provides learning lower bounds for noise sensitive functions, or classes of such functions. 
 
Multiple works have argued that the success of deep learning is intimately related to hierarchical properties in data generative processes.
This hypothesis has been independently proposed by multiple researchers, including:
\begin{itemize}
\item
Bruna and Mallat's wavelet scattering networks~\cite{BrunaMallat:13},
\item
Mhaskar, Liao, and Poggio's compositional function models~\cite{MhLiPo:16},
\item
Patel, Nguyen, and Baraniuk's hierarchical rendering models~\cite{PaNgBa:2015}.
\end{itemize}
Mossel~\cite{Mossel:19deep} provided a rigorous attempt to establish computational hardness, though subsequent work \cite{KoehlerMossel:22,HuangMossel:24} only demonstrated low-degree hardness in the analogue of the noise-sensitive regime (below the Kesten--Stigum bound).
In terms of compositions of functions, 
Telgarsky~\cite{Telgarsky:16} presents a notable result demonstrating an iteratively composed simple function that can be precisely expressed as a deep narrow network, yet requires super-polynomial computational complexity for small network depth.

The second research line establishes learning lower bounds in terms of the Fourier expansion of functions or classes of functions. 
This includes in particular SQ lower bounds in terms of the noise-sensitivity in both the Boolean case~\cite{dachman2014approximate} and the Gaussian case~\cite{diakonikolas2021optimality} as well as sample complexity lower bounds for agnostically learning such functions using fully connected neural networks and 
stochastic gradient descent~\cite{ABCKLRZ:22}.

We note that very recent work \cite{DPZK:25,RDKL:26} has formally shown that gradient descent is effective in learning the kind of functions proposed by Mossel \cite{Mossel:19deep} (which are not for product measures on the input), but these results are in the parameter regime which is analogous to noise-stability (i.e., above the Kesten--Stigum bound).

In our main result we show that hierarchical functions, i.e., functions obtained by decompositions of non-linear functions on disjoint independent inputs, are noise sensitive.
This  implies SQ lower bounds for agnostically learning such functions in both the Boolean and the Gaussian case as well as sample lower bounds for learning such functions using neural networks. 

Our work builds upon a long line of work in discrete Fourier analysis.
The fact that recursively composing the majority function on $3$ bits leads to a noise sensitive function is folklore in the area, see e.g.,~\cite{BenorLinial:90,BeKaSc:99}.
Mossel and O'Donnell~\cite{MosselODonnell:02} further study decompositions of monotone functions and construct such decompositions that are essentially ``most" noise sensitive among monotone functions. 
We also note that recursive constructions, i.e., Sipser functions, provide the best results in circuit lower bounds for circuits with AND, OR and NOT gates~\cite{HRRT:17}. 

The main contribution of the current paper is by showing that noise sensitivity of composition of non-linear functions on non-overlapping inputs is generic---the functions to be composed do not have to be identical, to be chosen carefully, or to be balanced or have any other property other than being somewhat uncorrelated with linear functions of their inputs. 
Our results hold for all product measures, both discrete and continuous.

\subsection{Our Results}
\subsubsection{Multilinear functions}
To begin with we consider the case where all functions involved are multilinear.
This includes in particular the case where all functions have binary inputs and outputs which is the classical case where analysis of Boolean functions and noise sensitivity are discussed.
In this case we obtain the following result concerning a single multilinear function.
Precise definitions are provided in \Cref{sec:multilinear}, though we comment that $d(f,\Lin)$ quantifies the non-linearity of $f$.
\begin{lemma}\label{lemma:multilinear}
    Suppose $\set{(X_i,Y_i):i\in[n]}$ are mutually independent, where for some $0\leq\rho\leq1$, we have $\abs{\Corr(X_i,Y_i)}\leq\rho$ for all $i$.
    Suppose $X_i$ and $Y_i$ have the same first and second moments, i.e., $\E[X_i]=\E[Y_i]$ and $\E[X_i^2]=\E[Y_i^2]<\infty$, for all $i\in[n]$.
    If $f:\R^n\to\R$ is a multilinear function satisfying $d(f,\Lin)\geq\varepsilon>0$, then 
    \begin{equation}\label{eq:lemma_multilinear}
        \abs{\Corr(f(X),f(Y))}\leq(1-\varepsilon)\rho+\varepsilon\rho^2.
    \end{equation}
\end{lemma}
\begin{remark}\label{remark:noise_stability} % Normally in Section 2.2.
    In the Boolean setting, i.e., when the $X_i$ and $Y_i$ are all Boolean random variables, when $\Corr(X_i,Y_i)=\rho$ for all $i$, then the quantity $\Corr(f(X),f(Y))$ that we study in \Cref{lemma:multilinear} coincides with the well-studied noise stability $\operatorname{Stab}_\rho(f)$.
    We refer readers to \cite{ODonnell:14} for a comprehensive overview of the analysis of Boolean functions.
    As we wish to state results that apply for general settings beyond just the Boolean setting, we will not use the specific definition of noise stability $\operatorname{Stab}_\rho(f)$, though the spirit remains the same. 
\end{remark}
\begin{remark}\label{remark:correlation_construction}
    Such a $Y$ can always be generated by setting $Y_i=X_i$ with probability $\rho$ and otherwise letting $Y_i$ be independent and identically distributed as $X_i$, independently for all $i$; note that this is not in general the unique construction.
    Moreover, this construction yields $X$ and $Y$ that have the same marginal distribution, which does not need to be true in general for \Cref{lemma:multilinear} to hold.
\end{remark}
Applying \Cref{lemma:multilinear} inductively yields the following result demonstrating exponential decay in noise stability for a hierarchical multilinear function, precisely defined in \Cref{subsec:hierarchical_multilinear}.
\begin{theorem}\label{theorem:multilinear}
    Suppose $\set{(X_i,Y_i):i\in[n]}$ are mutually independent, where for some $0\leq\rho<1$, we have $\abs{\Corr(X_i,Y_i)}\leq\rho$ for all $i$.
    Suppose $X_i$ and $Y_i$ have the same first and second moments, i.e., $\E[X_i]=\E[Y_i]$ and $\E[X_i^2]=\E[Y_i^2]<\infty$, for all $i\in[n]$.
    If $f:\R^n\to\R$ is a hierarchical multilinear function of non-linearity $\varepsilon>0$ and depth $d$, then for any $\delta<\varepsilon$,
    \[ \abs{\Corr(f(X),f(Y))} \leq \paren{1-\delta}^{d-C_{\varepsilon,\delta}\log\paren{\frac{1}{1-\rho}}} \]
    for some constant $C_{\varepsilon,\delta}>0$ depending only on $\varepsilon$ and $\delta$.
\end{theorem}
As discussed in \Cref{remark:multilinear_tightness}, this bound is essentially tight, as it is impossible to have exponential rate smaller than $1-\varepsilon$.
\Cref{remark:multilinear_epsilon1} discusses how this can be improved to doubly exponential decay in depth in the case where each component function has $0$ correlation with linear functions, such as $t$-resilient functions from cryptography, which are further discussed in \Cref{remark:multilinear_resilient}, and the parity bit function as further discussed in \Cref{remark:multilinear_parity}.
\subsubsection{General functions}
The results above leave something to be desired, as neural network activation functions, notably ReLU, are typically not multilinear in their inputs.
% It is also desirable to consider inputs that are real-valued and not just binary.
However, some care is needed in this general case as the following examples show.
\begin{example}\label{example:cos_arccos}
    Consider a hierarchical function with a binary tree structure, where at even levels we have component functions $f(x_1,x_2)=\cos(\pi x_1)$ and at odd levels $g(x_1,x_2)=\frac{\arccos(x_1)}{\pi}$, where each of the component variables takes values in $[0,1]$.
    For example, such a function of depth 2 would be given by $g(f(x_1,x_2),f(x_3,x_4))$.
    We note that $f$ and $g$ are clearly non-linear, but any such function of even depth will output $x_1$, which is as noise stable as possible, and there is no decay in correlation as depth increases.
\end{example}
\begin{example}\label{example:recursive_majority}
    The phenomenon above does not require the functions to have continuous inputs.
    In fact, it suffices for the inputs and outputs of the functions to take on a small number of values, e.g., four.
    For example, we can consider the input distributions to be uniform over $\set{-1,1}^{3^d}$, i.e., $3^d$ i.i.d. variables uniformly distributed on $\{-1,1\}$.
    In our correlated setup for noise sensitivity, we consider two such random vectors $X=(X_1,\dots,X_n)$ and $Y=(Y_1,\dots,Y_n)$ coupled such that $\{(X_i,Y_i)\}_{i=1}^n$ are mutually independent, and $\Corr(X_i,Y_i)=\rho$ for all $i$.

    In our construction, our functions can output at most four values, $\pm10\pm1$, and our overall hierarchical function has a ternary tree structure. 
    For a number $x=10b_1+b_2$ for $b_1,b_2\in\{-1,1\}$, let us write $B_1(x)=b_1$ and $B_2(x)=b_2$.
    We can then take the functions at the lowest layer (depth $d$) to be 
    \[ f(x_1,x_2,x_3)=x_1+10\maj(x_1,x_2,x_3), \]
    where $\maj(x_1,x_2,x_3)=2\cdot\1{x_1+x_2+x_3>0}-1$ is the majority function, i.e., takes value 1 if a majority of inputs take value 1, and takes value $-1$ if a majority of inputs takes value $-1$.
    We let the remaining functions be
    \[ g(x_1,x_2,x_3)=B_2(x_1)+10\maj(B_1(x_1),B_1(x_2),B_1(x_3)), \]
    %and the functions at all remaining levels be
    %\[ h(x_1,x_2,x_3)=B_2(x_1)+10\maj(B_1(x_1),B_1(x_2),B_1(x_3)). \]
    It is easy to see that all functions involved are far from linear, 
    %especially the non-root functions,
    but the hierarchical function on $3^d$ inputs computes  the first input plus $10$ times the recursive majority of all $3^d$ inputs, and is therefore very noise stable as it has constant correlation with $x_1$.
\end{example}

To account for potential non-linear relationships, we strengthen our notions of correlation in both our assumptions and our conclusions.
The precise definitions are provided in \Cref{sec:non_separable}, but loosely speaking, $\Corr(L^2(X),L^2(Y))$ denotes the maximum correlation between any two functions where one is a function of $X$ and the other is a function of $Y$.
The analog of \Cref{lemma:multilinear} in this more general setting is the following lemma.
\begin{lemma}\label{lemma:nonseparable}
Suppose $\set{(X_i,Y_i):i\in[n]}$ are mutually independent, where for some $0\leq\rho\leq1$, we have $\Corr(L^2(X_i),L^2(Y_i))\leq\rho$ for all $i$.
If $f:\R^n\to\R$ satisfies $\E[f(X)^2],\E[f(Y)^2]<\infty$ and
\begin{equation}\label{eq:lemma_nonseparable_condition}
    \Corr(L^2(f(X)),L^2(X_1)+\cdots+L^2(X_n)),\Corr(L^2(f(Y)),L^2(Y_1)+\cdots+L^2(Y_n))\leq\sqrt{1-\varepsilon}
\end{equation}
for some $\varepsilon>0$, then
\[ \Corr(L^2(f(X)),L^2(f(Y))) \leq (1-\varepsilon)\rho+\varepsilon\rho^2. \]
\end{lemma}
In \Cref{subsec:hierarchical_nonseparable}, we justify why \eqref{eq:lemma_nonseparable_condition} is the natural analog of the condition $d(f,\Lin)\geq\varepsilon$ in \Cref{lemma:multilinear}.
\begin{remark}\label{remark:correlation_construction_L2}
It is straightforward to check that the construction in \Cref{remark:correlation_construction} still satisfies the stronger condition $\Corr(L^2(X_i),L^2(Y_i))=\rho$: letting $Z_i$ be an independent copy of $X_i$ such that $Y_i=X_i$ with probability $\rho$ and $Y_i=Z_i$ otherwise, we have, for $f\in L^2(X_i)$ and $g\in L^2(Y_i)$ normalized so that $\E[f]=\E[g]=0$ and $\E[f^2]=\E[g^2]=1$, by the Cauchy--Schwarz inequality
\[ \Corr(f(X_i),g(Y_i)) = \rho\E[f(X_i)g(X_i)]+(1-\rho)\E[f(X_i)g(Z_i)] \leq \rho\cdot1 + (1-\rho)\cdot0 = \rho, \]
where equality is achieved when $f=g$ as functions $\R\to\R$.
\end{remark}
Using \Cref{lemma:nonseparable}, we obtain the following general version of \Cref{theorem:multilinear}.
\begin{theorem}\label{theorem:nonseparable}
    Suppose $\set{(X_i,Y_i):i\in[n]}$ are mutually independent, where for some $0\leq\rho<1$, we have $\Corr(L^2(X_i),L^2(Y_i))\leq\rho$ for all $i$.
    If $f:\R^n\to\R$ is a hierarchical function of depth $d$ and non-separability $\varepsilon>0$ in both $L^2(X)$ and $L^2(Y)$, then for any $\delta<\varepsilon$,
    \[ \Corr(L^2(f(X)),L^2(f(Y)))\leq (1-\delta)^{d-C_{\varepsilon,\delta}\log\paren{\frac{1}{1-\rho}}} \]
    for some constant $C_{\varepsilon,\delta}>0$ depending only on $\varepsilon$ and $\delta$.
\end{theorem}

\subsection{Applications}
\Cref{theorem:multilinear} immediately implies a low-degree bound on a hierarchical multilinear function---in particular we get that the correlation $M$ of such a function $f$ with a degree-$D$ polynomial is bounded by
\begin{equation}\label{eq:low_degree_multilinear}
    M \leq \rho^{-D/2}(1-\delta)^{d/2-C_{\varepsilon,\delta}\log\paren{\frac{1}{1-\rho}}},
\end{equation}
for any $\rho\in(0,1)$; see \Cref{lemma:low_degree_multilinear}.
Taking $\rho=1-e^{-cd}$ for some small $c\in\paren{0,\frac{1}{2C_{\varepsilon,\delta}}}$, we get a degree bound that is exponential in the circuit depth, i.e., in order to get $M=\Omega(1)$ correlation, one requires $D=e^{\Omega(d)}$.

In the Boolean setting, which is contained in our framework of multilinear functions, using the results of Dachman-Soled, Feldman, Tan, Wan and Wimmer~\cite{dachman2014approximate}, the bound \eqref{eq:low_degree_multilinear} implies super-polynomial lower bounds for Statistical Query (SQ) agnostic learning of concept classes based on hierarchical multilinear functions.
We refer readers to \cite{dachman2014approximate} for background on agnostic learning in the SQ model.
\begin{theorem}\label{theorem:sq_learning_Boolean}
    For any constant $\varepsilon>0$ and for any hierarchical multilinear function $f:\set{-1,1}^{n^{1/3}}\to\set{-1,1}$ of depth $d=\Theta(\log n)$ and non-linearity $\varepsilon$ that is balanced, i.e., $\E[f(X)]=0$ where $X\sim\Unif\paren{\set{-1,1}^{n^{1/3}}}$, let $\mc C$ be the concept class of functions $g:\set{-1,1}^n\to\set{-1,1}$ for which there exist distinct $i_1,\dots,i_{n^{1/3}}\in[n]$ such that $g(x)=f(x_{i_1},\dots,x_{i_{n^{1/3}}})$ for all $x$.
    Then any SQ algorithm for agnostically learning $\mc C$ with excess error of at most $\frac{1}{2}-o(1)$ has complexity at least $n^{\Omega\paren{\frac{\log n}{\log\log(n)}}}$.
\end{theorem}
\begin{remark}\label{remark:sq_learning_biased}
    We note that \Cref{theorem:sq_learning_Boolean} extends in a straightforward way to functions of arbitrary mean $\mu=\E[f(X)]$, with the excess error threshold $\frac{1}{2}-o(1)$ replaced by $\frac{1-\abs{\mu}}{2}-o(1)$.
    This threshold is optimal, as a single statistical query of constant tolerance estimates $\mu$, and the constant hypothesis $\operatorname{sign}(\mu)$ achieves excess error at most $\frac{1-\abs\mu}{2}$, so the mean, i.e., the level-0 Fourier coefficient, is always learnable in the SQ model.
    Proving this extension requires generalizing \cite[Theorem 3.2]{dachman2014approximate}, as well as the hardness result \cite[Theorem 1.1]{dachman2014approximate}, to a notion of resilience in which the constant Fourier coefficient is unconstrained.
    This extension is straightforward, but as this is not the focus of the present paper we provide a short sketch in \Cref{remark:biased_resilience}.
\end{remark}

It is natural to expect that similar results hold for other types of hierarchical functions, not necessarily as defined thus far.
A classical example of such functions is crossing events in critical percolation; we refer readers to \cite{Grimmett} for background on percolation.
Renormalization group arguments in statistical physics (see \cite{bauerschmidt2019renormalization} for an overview of renormalization groups and \cite{stauffer1992introduction} for a more focused treatment on percolation) have argued that such events are self-similar in the sense that the event views at different scales all look the same.
Thus, like the example of recursive majority of $3$, crossing events can be summarized (in some sense) over squares to obtain again a picture of critical percolation.
While this picture has not been proven in a rigorous way, the breakthrough results in statistical physics by Schramm, Werner, Smirnov and others (see, e.g., \cite{schramm2000sle,smirnov2001critical,werner2001critical}) allowed mathematicians to prove related detailed estimates about these functions and their Fourier spectra, which suffices for our noise sensitivity estimates.
In particular, applying a similar proof strategy as that of \cref{theorem:sq_learning_Boolean} to previously known results by Garban, Pete, and Schramm~\cite{garban2010fourier} on the noise sensitivity of critical site percolation on the triangular grid, we show percolation crossing events also create a concept class of functions that imply similar super-polynomial SQ lower bounds for agnostic learning.
\begin{theorem}\label{theorem:sq_learning_percolation}
    Let $f$ be the $\pm1$ indicator function of the left-to-right crossing event of the $n\times n$ rhombus, i.e., the parallelogram with side length $n$ whose sides are parallel to two of the lattice directions, in critical site percolation on the triangular grid.
    Let $N=\Theta(n^2)$ be the number of sites in this rhombus, so that $f$ is a function $\set{-1,1}^N\to\set{-1,1}$.
    Let $\mc C$ be the concept class of functions $g:\set{-1,1}^{N^3}\to\set{-1,1}$ for which there exist distinct $i_1,\dots,i_N\in\bracket{N^3}$ such that $g(x)=f(x_{i_1},\dots,x_{i_N})$ for all $x$.
    Then any SQ algorithm for agnostically learning $\mc C$ with excess error of at most $\frac{1}{2}-o(1)$ has complexity at least $n^{\Omega\paren{\frac{\log n}{\log\log(n)}}}$.
\end{theorem}
We state \Cref{theorem:sq_learning_percolation} for the rhombus because its crossing event is exactly balanced, i.e., has mean zero, at criticality.
The same result also holds for the left-to-right crossing event of the Euclidean square $[0,n]^2$, whose crossing event is nearly balanced; see \Cref{remark:percolation_square}.

Analogously, \Cref{theorem:nonseparable} immediately implies a low-degree bound on hierarchical functions, similar to \eqref{eq:low_degree_multilinear}.
The correlation $M$ of such a function $f$ with a function of Efron--Stein degree at most $D$ (see \Cref{subsec:Efron_Stein} for precise definitions), i.e., a sum of functions that each depend on at most $D$ of the $n$ variables, is bounded by
\begin{equation}\label{eq:low_degree_general}
    M \leq \rho^{-D/2}(1-\delta)^{d/2-C_{\varepsilon,\delta}\log\paren{\frac{1}{1-\rho}}};
\end{equation}
see \Cref{lemma:low_degree_general}.
As before, this implies a degree bound exponential in depth.

When our inputs are Gaussian, using the results of Diakonikolas, Kane, Pittas and Zarifis~\cite{diakonikolas2021optimality}, \Cref{theorem:nonseparable} implies super-exponential lower bounds for agnostic learning in the SQ model for concept classes based on hierarchical functions.
\begin{theorem}\label{theorem:sq_learning_Gaussian}
    Fix an arbitrary constant $\varepsilon>0$.
    Let positive integers $n$ and $m$ satisfy $m\leq n^\alpha$ for any constant $\alpha<\frac{1}{2}$.
    For any hierarchical function $f:\R^m\to\set{\pm1}$ of depth $d=\Theta(\log n)$ and non-separability $\varepsilon$ in $L^2(X)$, where $X\sim\mc N(0,I_m)$ is a $m$-dimensional spherical Gaussian of unit variance, let $\mc C$ be the concept class consisting of the functions of the form $F(x)=f(Px)$, where $P\in\R^{m\times n}$ is any matrix satisfying $PP^T=I_m$.
    Then there exists constant $c=c_\varepsilon>0$ such that any SQ learning algorithm that agnostically learns $\mc C$ over $\mc N(0,I_n)$ to $L^2$-error at most $\operatorname{OPT}+\frac{1-\E[f(X)]^2}{8}-o(1)$ either requires queries with tolerance at most $n^{-\Omega(n^c)}$ or makes at least $2^{n^{\Omega(1)}}$ queries.
\end{theorem}
Our final application is for learning such functions using stochastic gradient descent on fully connected deep neural networks. 
The results of~\cite{ABCKLRZ:22} which build on~\cite{AbbeSandon:20} imply that for functions as above, in their setting of fully connected neural networks and their initialization, the complexity of learning the function using their version of gradient descent is exponential in the depth of the circuit.
In the special case where all functions are resilient, i.e., have $0$ correlation with linear function, the lower bound becomes doubly exponential in the depth, i.e., exponential in the input size.
More formally, \cite[Theorem 1]{ABCKLRZ:22} immediately implies the following.
\begin{corollary}
Let $f : \{-1,1\}^n \to \{-1,1\}$ be a hierarchical multilinear function of non-linearity $\eps>0$ and depth $d = \Omega(\log n)$ that is balanced, i.e., $\E[f(X)]=0$ where $X\sim\Unif(\set{-1,1}^n)$.
Let $\bar{f}$ be the $2n$-extension of $f$, given by 
$\bar{f}(x_1,\ldots,x_{2n}) = f(x_1,\ldots,x_n)$.
Consider a fully connected neural network of size $E$ with initialization that is target
agnostic in the sense that it is invariant under permutations of the inputs. 
Let $f^{(t)}_{\mathrm{NN}}$ denote the output of noisy-GD with gradient range $A$, batch-size $b$, learning rate $\gamma$, and noise scale $\sigma$ (see \cite[Definition 8]{ABCKLRZ:22}) after $t$ time steps. Then, if $1 - \rho \leq 0.25$  and $b$ is sufficiently large, then for any $\delta < \eps$ the generalization error is at least 
$
0.5 - o(1)$ as long as 
\[
\gamma t \sqrt{E} A (1-\delta)^{d-C_{\eps,\delta} \log(\frac{1}{1-\rho})}  \ll \sigma. 
\] 
In the case where $\eps = 1$, we get generalization error at least $0.5 - o(1)$ as long as 
$\gamma t \sqrt{E} A \rho^{2^{d-2}} \ll \sigma$.
\end{corollary}
Note that the bound in the case $\eps < 1$ is polynomial in $n$, while we get super-polynomial lower bound in the case where $\eps = 1$ (see  \cref{remark:multilinear_epsilon1} below).

We conjecture that the lower bounds of Abbe, Bengio, Cornacchiam, Kleinberg, Lotfi, Raghu and Zhang~\cite{ABCKLRZ:22} should extend in a straightforward way to other product spaces.

We finally note briefly a different set of results involving noise sensitivity and deep neural networks, i.e., the results of \cite{jonasson2023noise}, which studied the noise sensitivity of ReLU networks with i.i.d.\ Gaussian weights.

\subsection{Section Overview}
In \Cref{sec:multilinear}, we introduce the necessary definitions to prove the results concerning multilinear functions.
Then in \Cref{sec:non_separable}, we first introduce the necessary formalization of correlated probability spaces in \Cref{subsec:correlated_spaces}, then introduce the Efron--Stein decomposition in \Cref{subsec:Efron_Stein} and its fundamental properties, especially in relation to Markov operators; finally, properly equipped, we can prove the results concerning general functions.

\section{Noise stability of compositions of multilinear functions}\label{sec:multilinear}
\subsection{Preliminary setup}
For any probability space $(\Omega,\P)$, function $f\in L^2(\Omega,\P)$, and subset $V\subseteq L^2(\Omega,\P)$, define the (squared) distance between $f$ and $V$ by
\[ d(f,V) = \inf_{g\in V}\frac{\E[(f-g)^2]}{\Var(f)}. \]
If $f$ is almost surely constant, i.e., $\Var(f)=0$, we define $d(f,V)=0$.
Similarly, throughout this paper we use the convention $\Corr(X,Y)=0$ if either $X$ or $Y$ are almost surely constant.
We now specialize to the case where $\P$ is a product measure on product space $\Omega$ over laws of real-valued random variables $X_1,X_2,\dots,X_n$.
In the product measure, these random variables are mutually independent, and we additionally assume the random variables have finite second moment, i.e., $\E[X_i^2]<\infty$ for all $i\in[n]$.
Beyond this, we make no further assumptions on the distributions of the random variables, e.g., some may be discrete, others continuous, or neither.
We will consider the setting where $f\in L^2(\Omega,\P)$ is multilinear, i.e., $f$ is of the form
\[ f(x_1,\dots,x_n) = \sum_{S\subseteq[n]}a_S \prod_{i\in S}x_i \]
for some coefficients $a_S\in\R$.
Note that any function of the above form is in $L^2(\Omega,\P)$.
In the Boolean case, i.e., when $X_1,\dots,X_n$ are each random variables supported on at most two points, note that all $f\in L^2(\Omega,\P)$ are multilinear in $X_1,\dots,X_n$.

Let $\Lin$ denote the subspace of $L^2(\Omega,\P)$ spanned by linear functions, i.e.,
\[ \Lin=\operatorname{span}(\{1,x_1,x_2,\dots,x_n\}). \]
For each $i\in[n]$, define the normalized random variable $\widetilde X_i = \frac{X_i-\E[X_i]}{\sqrt{\Var(X_i)}}$, which has mean 0 and variance 1, and similarly denote $\widetilde x_i = \frac{x_i-\E[X_i]}{\sqrt{\Var(X_i)}}$ for all $x_i\in\R$.
Then for any multilinear $f$, we let $\widehat f(S)\in\R$ for $S\subseteq[n]$ denote the unique coefficients satisfying
\[ f(x_1,\dots,x_n) = \sum_{S\subseteq[n]}\widehat f(S)\prod_{i\in S}\widetilde x_i. \]
For each $S\subseteq[n]$, we let $\chi_S:\R^n\to\R$ denote the function given by
\[ \chi_S(x_1,\dots,x_n) = \prod_{i\in S}\widetilde x_i. \]
By independence of $X_1,\dots,X_n$, we have $\set{\chi_S:S\subseteq[n]}$ is a set of orthonormal multilinear polynomials: for distinct $S,S'\subseteq[n]$, we have $\E[\chi_S\chi_{S'}]=0$ and $\E[\chi_S^2]=1$.
Also note that $\E[\chi_S]=\1{S=\emptyset}$.
Hence, we have
\[ \E[f^2] = \sum_{S\subseteq[n]}\widehat f(S)^2 \quad\text{and}\quad \Var(f) = \sum_{S\neq\emptyset}\widehat f(S)^2. \]
From this, it is easy to see that
\[ d(f,\Lin) = \frac{\displaystyle\sum_{\abs{S}\geq2}\widehat f(S)^2}{\displaystyle\sum_{S\neq\emptyset}\widehat f(S)^2}. \]
For any $S\subseteq[n]$, we let $X_S$ denote the random vector consisting of the $X_i$ for $i\in S$, and similarly define $Y_S$, as well as $x_S$ and $y_S$ for $x,y\in\R^n$.
For example, $\chi_S(x)$ depends only on $x_S$.
\subsection{Main lemma}
\Cref{lemma:multilinear} shows that if we have some noisy observation $Y_i$ of each $X_i$, namely $\Corr(X_i,Y_i)\leq\rho$ for all $i\in[n]$ for some constant $\rho<1$, and $f$ is far from linear, then $\Corr(f(X),f(Y))<\rho$, with this correlation displaying exponential decay with respect to application of $f$.
We now prove this lemma.
\begin{proof}[Proof of \Cref{lemma:multilinear}]
    Note that because $X_1,\dots,X_n$ are mutually independent, and $Y_1,\dots,Y_n$ are also mutually independent, we have $\Var(f(X))=\Var(f(Y))$ as 
    \[ \E[\chi_S(X)\chi_{S'}(X)]=\E[\chi_S(Y)\chi_{S'}(Y)]=\1{S=S'}. \]
    We similarly also have $\E[f(X)]=\E[f(Y)]=\widehat f(0)$ as $\E[\chi_S(X)]=\E[\chi_S(Y)]=\1{S=\emptyset}$.
    Thus, without loss of generality, we can rescale and shift $f$ so that $\Var(f)=1$ and $\E[f]=0$.
    Then we expand
    \begin{align*}
        \Corr(f(X),f(Y))
        &= \E[f(X)f(Y)]
        = \sum_{S\neq\emptyset}\sum_{S'\neq\emptyset}\widehat f(S)\widehat f(S')\E[\chi_S(X)\chi_{S'}(Y)].
    \end{align*}
    Note that if $S\neq S'$, we have $\E[\chi_S(X)\chi_{S'}(Y)]=0$.
    Then either $S\not\subseteq S'$ or $S'\not\subseteq S$.
    Say $S\not\subseteq S'$.
    Then
    \begin{align*}
        \E[\chi_S(X)\chi_{S'}(Y)]
        &=\E[\E[\chi_S(X)\chi_{S'}(Y)|(X_{S'},Y_{S'})]]
        \\ &=\E[\E[\chi_S(X)|(X_{S'},Y_{S'})]\chi_{S'}(Y)]
        \\ &=\E[0\cdot\chi_{S'}(Y)]
        \\ &=0.
    \end{align*}
    The case $S'\not\subseteq S$ follows symmetrically.
    If $S=S'$, then by independence of the $(X_i,Y_i)$ we have
    \begin{align*}
        \abs{\E[\chi_S(X)\chi_S(Y)]}
        &= \abs{\prod_{i\in S}\E[\widetilde X_i \widetilde Y_i]}
        = \abs{\prod_{i\in S}\Corr(X_i,Y_i)}
        \leq \rho^{\abs{S}}.
    \end{align*}
    Thus,
    \[ \abs{\Corr(f(X),f(Y))} \leq \sum_{S\neq\emptyset}\rho^{\abs{S}}\widehat f(S)^2. \]
    Recalling that
    \[ \Var(f) = \sum_{S\neq\emptyset} \widehat f(S)^2 = 1 \]
    while
    \[ d(f,\Lin)=\sum_{\abs{S}\geq2}\widehat f(S)^2 \geq\varepsilon,\]
    we obtain the bound
    \[ \abs{\Corr(f(X),f(Y))} \leq (1-\varepsilon)\rho+\varepsilon\rho^2, \]
    as desired.
\end{proof}
Note that this bound is sharp when the conditions in the statement of the lemma are tight, i.e., $\Corr(X_i,Y_i)=\rho$ for all $i$, and $d(f,\Lin)=\varepsilon$, and $f$ is of degree 2.
\subsection{Hierarchical multilinear functions}\label{subsec:hierarchical_multilinear}
We now consider functions consisting of a composition of multilinear functions in a hierarchical structure.
Specifically, for any positive integers $n$ and $d$, we define a function $f:\R^n\to\R$ to be a \dfn{hierarchical multilinear function} of non-linearity $\varepsilon>0$ and depth $d$ recursively as follows.
A hierarchical multilinear function of non-linearity $\varepsilon>0$ and depth 1 is simply a function satisfying the conditions of \Cref{lemma:multilinear}, given mutually independent pairs of random variables $\set{(X_i,Y_i):i\in[n]}$.
Note $d(f,\Lin)$ depends on the underlying random variables only through their first two moments, so ensuring these are the same for $X_1,\dots,X_n$ and $Y_1,\dots,Y_n$ as in \Cref{lemma:multilinear} yields the same value for $d(f,\Lin)$ regardless of whether we view $f$ as a function of $X_1,\dots,X_n$ or $Y_1,\dots,Y_n$. 
A hierarchical multilinear function of non-linearity $\varepsilon>0$ and depth $d\geq2$ is a function $f:\R^n\to\R$ of the form
\[ f(x) = g(h_1(x_{\mc P_1}),\dots,h_m(x_{\mc P_m}))\]
for some partition $\mc P$ of $[n]$ with $m\geq 2$ parts and hierarchical multilinear functions $g,h_1,\dots,h_m$ of non-linearity $\varepsilon$, where $g$ has depth 1 and each $h_i$ has depth at least $d-1$.
Here, nonlinearity is defined for each function using the distributions of its inputs, e.g., $g$ is viewed as a function of $h_1(X_{\mc P_1}),\dots,h_m(X_{\mc P_m})$.
Fixing the first two moments of the inputs $X_1,\dots,X_n$ as in \Cref{theorem:multilinear} also fixes the first two moments of all intermediate outputs in the hierarchy, e.g., $h_1(X_{\mc P_1})$, because each component function is multilinear; thus, the nonlinearity of a hierarchical multilinear function $f:\R^n\to\R$ depends on the distribution of the underlying random variables $X_1,\dots,X_n$ only through their first two moments, e.g., is consistently defined regardless of whether we view $f$ as a function of $X_1,\dots,X_n$ or $Y_1,\dots,Y_n$ in the setting of \Cref{theorem:multilinear}.
While this definition allows for the $h_i$ to have depth greater than or equal to $d$, i.e., the tree need not have all leaves at the same depth, note that any depth-$d$ hierarchical multilinear function of non-linearity $\varepsilon>0$ must be a function of at least $2^d$ variables, so our recursive definition does terminate and is well-defined.
Equivalently, any hierarchical multilinear function of $n$ variables has depth at most $\log_2 n$.

As previously mentioned, all functions on Boolean random variables are multilinear, so any function on Boolean random variables of this hierarchical form whose component functions all have Boolean outputs, i.e., take on two values, except for potentially the top component function, are hierarchical multilinear functions.
Thus, our definition of a hierarchical multilinear function includes all hierarchical Boolean functions.

We can now prove \Cref{theorem:multilinear}, which shows that $\Corr(f(X),f(Y))$ decays exponentially in depth $d$, for hierarchical multilinear functions of positive non-linearity.
\begin{proof}[Proof of \Cref{theorem:multilinear}]
    Let $g(x)=(1-\varepsilon)x+\varepsilon x^2$.
    For positive integer $m\geq 1$, let $g^m$ denote the composition of $g$ a total of $m$ times.
    Because $X_i$ and $Y_i$ have the same first and second moments, and because the $(X_i,Y_i)$ are mutually independent across $i$, any multilinear function $h$ will have $\E[h(X)]=\E[h(Y)]$ and $\E[h(X)^2]=\E[h(Y)^2]$, and thus we can repeatedly apply \Cref{lemma:multilinear} for all component functions of our hierarchical multilinear function $f$, as the first and second moment conditions of \Cref{lemma:multilinear} inductively hold.
    By repeatedly applying \Cref{lemma:multilinear} throughout the hierarchical structure of $f$, and noting that $g$ is an increasing function on $[0,1]$ with $0 \leq g(x)<x$ for all $0 \leq x < 1$, we have $\abs{\Corr(f(X),f(Y))}\leq g^d(\rho)$, so it suffices to show $g^d(\rho)\leq \paren{1-\delta}^{d-C_{\varepsilon,\delta}\log\paren{\frac{1}{1-\rho}}}$ for all $d\geq 1$ and $0\leq\rho<1$ for some constant $C_{\varepsilon,\delta}$.
    Note that $\frac{g(x)}{x}=1-\varepsilon+\varepsilon x$, so there exists some $\alpha=\alpha_{\varepsilon,\delta}\in(0,1)$ such that for all $x\leq\alpha$, we have $g(x)\leq(1-\delta)x$.
    Thus it suffices to show there exists some nonnegative integer $N\leq C_{\varepsilon,\delta}\log\paren{\frac{1}{1-\rho}}$ such that $g^{N}(\rho)\leq\alpha$.
    If $\rho\leq\alpha$, we simply take $N=0$.
    Otherwise, note that for $\alpha\leq x<1$, we have
    \[ \frac{1-g(x)}{1-x} = 1+\varepsilon - \varepsilon(1-x) \geq 1+\varepsilon\alpha>1,\]
    so we may take $N\geq\log_{1+\varepsilon\alpha}\paren{\frac{1-\alpha}{1-\rho}}.$
\end{proof}
\begin{remark}\label{remark:multilinear_tightness}
    We remark that this bound is essentially tight.
    Suppose all component functions defining our hierarchical multilinear function $f$ have non-linearity exactly $\varepsilon>0$, i.e., each component function $g$ satisfies $d(g,\Lin)=\varepsilon$, and $\Corr(X_i,Y_i)=\rho$ for all $i$.
    With these tight conditions, following the proof of \Cref{lemma:multilinear}, it is straightforward to see that we could have proven $\Corr(f(X),f(Y))\geq(1-\varepsilon)\rho$ instead of \eqref{eq:lemma_multilinear}.
    Inductively applying this bound instead of \Cref{lemma:multilinear}, we see that our hierarchical multilinear function of depth $d$ has $\Corr(f(X),f(Y))\geq(1-\varepsilon)^d\rho$.
    Hence, it is impossible to have exponential rate smaller than $1-\varepsilon$, and we have provided a bound with exponential rate $\alpha$ for all $\alpha>1-\varepsilon$.
\end{remark}
\begin{remark}\label{remark:multilinear_epsilon1}
    Note that in the case $\varepsilon=1$, i.e., $f$ is a hierarchical multilinear function where each component function is completely uncorrelated with any linear function, \Cref{lemma:multilinear} actually implies $\abs{\Corr(f(X),f(Y))}\leq\rho^{2^d}$.
    In this case, we have doubly exponential decay in depth, with decay rate depending on $\rho$, stronger than the singly exponential decay provided by \Cref{theorem:multilinear}, where the decay rate depends on $\varepsilon$.
\end{remark}
\begin{remark}\label{remark:multilinear_resilient}
    In the literature of cryptography, a Boolean function is \dfn{$t$-resilient} if it is balanced and remains unpredictable even if an adversary knows up to $t$ of the $n$ input bits.
    Concretely, this means $\widehat f(S)=0$ for all $\abs{S}\leq t$.
    We refer interested readers to \cite[Chapter 8]{CramaHammer} for further background on Boolean functions in cryptography.
    Our results do not care about the balanced condition, i.e., $\widehat f(\emptyset)=0$, but aside from this, our case $\varepsilon=1$ in the Boolean setting corresponds to 1-resilient functions.
    It is straightforward to see, by adapting the proof of \Cref{lemma:multilinear}, that for any positive integer $t$, if we have a hierarchical function whose components are all $t$-resilient Boolean functions, then we have $\abs{\Corr(f(X),f(Y))}\leq\rho^{(t+1)^d}$.
\end{remark}
\begin{remark}\label{remark:multilinear_parity}
    One example of such a component function is the parity function in the Boolean case, i.e., $\operatorname{Parity}(x)=x_1\cdot x_2\cdots x_n$, where $X_i\sim\operatorname{Unif}(\set{\pm1})$ for all $i$.
    This function takes value 1 if an even number of bits are on, i.e., have value $-1$, and value $-1$ if an odd number of bits are on.
    In the coding theory literature, this function is more commonly written as
    \[ \operatorname{Parity}(x_1,\dots,x_n)=x_1\oplus\cdots\oplus x_n, \]
    where $\oplus$ denotes the XOR (exclusive OR), and each $x_i\in\{0,1\}$ is a bit.
    However, to fit this into our multilinear polynomial framework, we convert the additive group $\{0,1\}$ modulo 2, i.e., $\Z/2\Z$, into the multiplicative group $\{1,-1\}$, which is a standard transformation in the literature of analysis of Boolean functions.
    Assuming $\Corr(X_i,Y_i)=\rho$ for all $i$, it is well-known, and can be seen by adapting the proof of \Cref{lemma:multilinear}, that $\Corr(\operatorname{Parity}(X),\operatorname{Parity}(Y))=\rho^n$, i.e., exponential in $n$.
    Note that one can view the parity function as a hierarchical function of smaller component parity functions, and by expressing the $n$-bit parity function as a hierarchical function of depth $\floor{\log_2n}$ with components being 2-bit parity functions, we have, via the $\varepsilon=1$ version of \Cref{theorem:multilinear} discussed in \Cref{remark:multilinear_epsilon1}, that 
    \[ \abs{\Corr(\operatorname{Parity}(X),\operatorname{Parity}(Y))} \leq \rho^{2^{\floor{\log_2n}}} < \rho^{n/2}, \]
    so our result still demonstrates exponential decay in $n$, albeit with a different rate than the true rate.
\end{remark}
For example, note that \Cref{theorem:multilinear} shows that the noise stability of the Boolean function on $3^d$ bits consisting of recursively composing the majority function on 3 bits, which is not linear, decays exponentially in $d$.
Thus, the recursive majority function is noise sensitive for logarithmic depth. 
Recall that the fact that the recursive majority function is noise sensitive is folklore; see, for example, \cite{BenorLinial:90,BeKaSc:99}.

\Cref{theorem:multilinear} implies low-degree bounds on hierarchical multilinear functions.
To make this concrete, the following lemma shows how bounds on noise stability, that is, $\Corr(f(X),f(Y))$, imply bounds on the maximum correlation between $f$ and any multilinear function of degree at most $D$. 
\begin{lemma}\label{lemma:low_degree_multilinear}
    Suppose $\{(X_i,Y_i):i\in[n]\}$ are mutually independent, where for some $0\leq\rho\leq1$, we have $\Corr(X_i,Y_i)=\rho$ for all $i$.
    Suppose $X_i$ and $Y_i$ have the same first and second moments, i.e., $\E[X_i]=\E[Y_i]$ and $\E[X_i^2]=\E[Y_i^2]<\infty$, for all $i\in[n]$.
    Let $f:\R^n\to\R$ be a multilinear function, and let $M$ denote the maximum correlation between $f$ and any multilinear function of degree at most $D$, i.e.,
    \[ M = \sup_{\substack{g:\R^n\to\R\text{ multilinear}\\\deg(g)\leq D}} \Corr(f(X),g(X)). \]
    Then
    \[ M \leq \rho^{-D/2}\sqrt{\Corr(f(X),f(Y))}. \]
\end{lemma}
\begin{proof}
    We wish to bound
    \[ M = \sup_{\substack{g:\R^n\to\R\text{ multilinear}\\ \deg(g)\leq D}}\Corr(f(X),g(X)). \]
    For any such $g$, decompose $f$ and $g$ as
    \[ f = \sum_{S\subseteq[n]}\widehat f(S)\chi_S, \quad g = \sum_{S\subseteq[n]}\widehat g(S)\chi_S, \]
    and without loss of generality rescale $g$ so that $\Var(g)=1$, i.e.,
    \[ \sum_{S\neq\emptyset}\widehat g(S)^2 = 1. \]
    Following the proof of \Cref{lemma:multilinear}, we know
    \[ \Corr(f(X),f(Y)) = \frac{1}{\Var(f)}\sum_{S\neq\emptyset}\rho^{\abs{S}}\widehat f(S)^2. \]
    Similarly, i.e., via orthogonality of $\set{\chi_S:S\subseteq[n]}$, by the Cauchy--Schwarz inequality, we have
    \begin{align*}
        \Corr(f(X),g(X))
        = \frac{1}{\sqrt{\Var(f)}}\sum_{S\neq\emptyset}\widehat f(S)\widehat g(S)
        = \frac{1}{\sqrt{\Var(f)}}\sum_{1 \leq \abs{S}\leq D}\widehat f(S)\widehat g(S)
        \leq \frac{\sqrt{\displaystyle\sum_{1\leq\abs{S}\leq D}\widehat f(S)^2}}{\sqrt{\Var(f)}},
    \end{align*}
    with equality if and only if
    \[ g = C+\frac{\displaystyle\sum_{1\leq\abs{S}\leq D}\widehat f(S)\chi_S}{\sqrt{\displaystyle\sum_{1\leq\abs{S}\leq D}\widehat f(S)^2}}\]
    for some constant $C\in\R$.
    Thus,
    \[ M = \frac{1}{\sqrt{\Var(f)}}\sqrt{\sum_{1\leq\abs{S}\leq D}\widehat f(S)^2}. \]
    Then we relate $\Corr(f(X),f(Y))$ with $M$ via
    \begin{align*}
        \Corr(f(X),f(Y))
        &= \frac{1}{\Var(f)}\sum_{S\neq\emptyset}\rho^{\abs{S}}\widehat f(S)^2
        \\ &\geq \frac{1}{\Var(f)}\sum_{1\leq\abs{S}\leq D}\rho^{\abs{S}}\widehat f(S)^2
        \\ &\geq \frac{1}{\Var(f)}\sum_{1\leq\abs{S}\leq D}\rho^D\widehat f(S)^2
        \\ &= \rho^D M^2,
    \end{align*}
    which rearranges to our desired inequality.
\end{proof}
Using \Cref{theorem:multilinear}, for any hierarchical multilinear function $f$ of non-linearity $\varepsilon>0$ and depth $d$, and for any $\delta<\varepsilon$ and $0 \leq \rho < 1$, applying \Cref{lemma:low_degree_multilinear} yields \eqref{eq:low_degree_multilinear}.
Note that this is because, as previously discussed in \Cref{remark:correlation_construction}, we can always construct a $\rho$-correlated $Y$ satisfying the conditions of \Cref{theorem:multilinear}.
In particular, for fixed $D$, our bound \eqref{eq:low_degree_multilinear} decays exponentially in depth $d$.

To prove \Cref{theorem:sq_learning_Boolean}, we recall the relevant notions of resilience from \cite{dachman2014approximate}, stated for functions on $k$ input bits, where all Fourier coefficients and $L^p$ norms are with respect to the uniform distribution on $\set{-1,1}^k$.
\begin{definition}[{\cite[Definition 1.8]{dachman2014approximate}}]\label{definition:approximate_resilience}
    A function $g:\set{-1,1}^k\to\R$ is \dfn{$D$-resilient} if $\widehat g(S)=0$ for all $\abs{S}\leq D$; note that, in particular, a $D$-resilient function has mean zero.
    A Boolean function $f:\set{-1,1}^k\to\set{-1,1}$ is \dfn{$\alpha$-approximately $D$-resilient} if there exists a $D$-resilient function $g:\set{-1,1}^k\to[-1,1]$ such that $\norm{f-g}_1=\E\bracket{\abs{f-g}}\leq\alpha$.
\end{definition}
The result of \cite{dachman2014approximate} that we use to verify approximate resilience is the following.
\begin{theorem}[{\cite[Theorem 3.2]{dachman2014approximate}}]\label{theorem:approximate_resilience}
    There exists a universal constant $K>0$ such that the following holds.
    Let $f:\set{-1,1}^k\to\set{-1,1}$ satisfy $\sum_{\abs{S}\leq D}\widehat f(S)^2\leq\gamma$ for some $D\in[k]$ and $\gamma\in[0,1]$.
    Then for all $\tau>e^D\sqrt{\gamma}$, we have that $f$ is $O\paren{\tau+\delta k^{2D+2}}$-approximately $D$-resilient, where $\delta=\exp\paren{-K(\tau^2/\gamma)^{1/D}}$.
\end{theorem}
Approximate resilience implies hardness of agnostic learning in the SQ model via \cite[Theorem 1.1]{dachman2014approximate}, which when applied to our setting will immediately yield \cref{theorem:sq_learning_Boolean,theorem:sq_learning_percolation}, as follows.
\begin{proof}[Proof of \Cref{theorem:sq_learning_Boolean}]
    From the proof of \Cref{lemma:low_degree_multilinear}, we have for any $D\in[n]$,
    \[ M = \frac{1}{\sqrt{\Var(f)}}\sqrt{\sum_{1\leq\abs{S}\leq D}\widehat f(S)^2} \geq \sqrt{\sum_{1\leq\abs{S}\leq D}\widehat f(S)^2} \]
    as $f(x)\in[-1,1]$ for all $x$.
    Then by \eqref{eq:low_degree_multilinear}, we have, fixing $\rho\in(0,1)$ to be a constant and letting $\gamma=\sum_{\abs{S}\leq D}\widehat f(S)^2=\sum_{1\leq\abs{S}\leq D}\widehat f(S)^2$ as $\widehat f(\emptyset)=\E[f(X)]=0$,
    \[ \sqrt{\gamma}=\sqrt{\sum_{1\leq\abs{S}\leq D}\widehat f(S)^2} \leq C_{\rho,\varepsilon,\delta} \rho^{-D/2}(1-\delta)^{d/2} = e^{\Theta(D)-\Theta(\log n)}, \]
    for any constant $\delta\in(0,\eps)$ and some constant $C_{\rho,\varepsilon,\delta}>0$.
    Pick constants $c_1,c_2>0$ sufficiently small with $c_1\leq c_2$ so that, letting $D=\frac{c_1\log n}{\log\log n}$ and $x=n^{c_2}$, where $x> e^D$ for sufficiently large $n$, we have
    \[ xC_{\rho,\varepsilon,\delta} \rho^{-D/2}(1-\delta)^{d/2} = o(1). \]
    Let this expression be denoted by $\tau$, and note that $\tau\geq x\sqrt{\gamma}>e^D\sqrt{\gamma}$.
    \Cref{theorem:approximate_resilience} implies there exists constant $K>0$ such that $f$ is $O\paren{\tau+\exp\paren{-Kx^{2/D}}n^{\frac23D+\frac23}}$-approximately $D$-resilient, where we used that $\tau^2/\gamma\geq x^2$ and that $f$ depends on $n^{1/3}$ coordinates.
    Note that, re-picking $c_1$ as necessary, this asymptotic expression is $o(1)$, because
    \[ -Kx^{2/D}+\paren{\frac23D+\frac23}\log n= -K(\log n)^{\frac{2c_2}{c_1}}+\Theta\paren{\frac{\log^2(n)}{\log\log(n)}} =-\omega(1). \]
    Applying \cite[Theorem 1.1]{dachman2014approximate} then completes the proof.
\end{proof}
A similar proof yields \Cref{theorem:sq_learning_percolation}.
\begin{proof}[Proof of \Cref{theorem:sq_learning_percolation}]
    First, we observe that $f$ is exactly balanced, i.e., $\widehat f(\emptyset)=\E[f]=0$; this is a well-known property of critical percolation on the triangular lattice (see, e.g., \cite[Section 1.2]{werner2009lectures}).
    Indeed, on the triangular lattice, the complement of the event that the rhombus has a left-to-right open crossing is precisely the event that it has a top-to-bottom closed crossing.
    At criticality $p=\frac{1}{2}$, composing the reflection of the rhombus that exchanges its two pairs of sides with the swap of the colors open and closed yields a measure-preserving involution exchanging these two events, so each has probability exactly $\frac{1}{2}$.
    Garban, Pete, and Schramm~\cite{garban2010fourier}
    showed that
    \[ \gamma = \sum_{0 \leq \abs{S}\leq D}\widehat f(S)^2 = \sum_{1\leq\abs{S}\leq D}\widehat f(S)^2 = n^{-\frac{1}{2}+o(1)}D^{\frac{2}{3}+o(1)} \]
    for any $D\in\bracket{N^3}$; their estimates apply to crossing events of general quads of the triangular grid, and in particular to the rhombus.
    Let $D=\frac{c\log n}{\log\log n}$ and $x=n^c$ for some constant $c\in\paren{0,\frac{1}{4}}$.
    Then we have
    \[ x\sqrt{\gamma}=n^cn^{-\frac{1}{4}+o(1)}=o(1) \]
    as $c<\frac{1}{4}$.
    Let this expression be denoted by $\tau$, and note that $\tau=x\sqrt{\gamma}>e^D\sqrt{\gamma}$ for sufficiently large $n$.
    Then \Cref{theorem:approximate_resilience} implies there exists constant $K>0$ such that $f$ is $O\paren{\tau+\exp\paren{-Kx^{2/D}}N^{3(2D+2)}}$-approximately $D$-resilient, where we used that $\tau^2/\gamma=x^2$.
    This asymptotic expression is $o(1)$ as $\tau=o(1)$ and
    \[ -Kx^{2/D}+(12D+12)\log n = -K(\log n)^2 + \Theta\paren{\frac{\log^2(n)}{\log\log(n)}}=-\omega(1). \]
    Applying \cite[Theorem 1.1]{dachman2014approximate} then completes the proof.
\end{proof}
\begin{remark}\label{remark:percolation_square}
    \Cref{theorem:sq_learning_percolation} also holds for the left-to-right crossing event of the Euclidean square $[0,n]^2$, the primary setting of \cite{garban2010fourier}.
    Unlike the rhombus, the square is not exactly balanced at any finite $n$; however, the crossing probability of the square converges as $n\to\infty$ to its Cardy--Smirnov limit \cite{smirnov2001critical}, which equals $\frac{1}{2}$ for the square by symmetry, and this convergence holds at a power-law rate: by the results of Mendelson, Nachmias, and Watson~\cite{mendelson2014rate} (see also Binder, Chayes, and Lei~\cite{binder2015rate}), there exists a constant $\varepsilon_0>0$ such that the crossing probability of the square is $\frac{1}{2}+O(n^{-\varepsilon_0})$, and hence $\abs{\widehat f(\emptyset)}=O(n^{-\varepsilon_0})$.
    Choosing the constant $c$ in the proof of \Cref{theorem:sq_learning_percolation} to additionally satisfy $c<\varepsilon_0$, the mean term is then absorbed exactly as in the proof of \Cref{theorem:sq_learning_Boolean}.
\end{remark}
\begin{remark}\label{remark:biased_resilience}
    We now sketch the generalization of \cite{dachman2014approximate} required for the extension of \Cref{theorem:sq_learning_Boolean} to functions of arbitrary mean described in \Cref{remark:sq_learning_biased}.
    Say that $g:\set{-1,1}^k\to[-1,1]$ is \dfn{$D$-resilient off the constant} if $\widehat g(S)=0$ for all $1\leq\abs{S}\leq D$, with $\widehat g(\emptyset)$ unconstrained, and that a Boolean function $f:\set{-1,1}^k\to\set{-1,1}$ is \dfn{$\alpha$-approximately $D$-resilient off the constant} if $\norm{f-g}_1\leq\alpha$ for some such $g$; this is \Cref{definition:approximate_resilience} with the constraint on the degree-$0$ coefficient dropped.
    The relevant generalization of \Cref{theorem:approximate_resilience} states that its conclusion holds, with resilience replaced by resilience off the constant, under the weaker hypothesis $\sum_{1\leq\abs{S}\leq D}\widehat f(S)^2\leq\gamma$, i.e., with no assumption on the mean of $f$.
    The proof of \cite[Theorem 3.2]{dachman2014approximate} adapts as follows.
    One runs the argument with the mean-free low-degree part $\ell=\sum_{1\leq\abs{S}\leq D}\widehat f(S)\chi_S$ in place of the full low-degree part of $f$: the concentration inequality for low-degree polynomials \cite[Theorem 3.3]{dachman2014approximate} applies to $\ell$ unchanged, as $\ell$ is a degree-$D$ polynomial with $\norm{\ell}_2^2\leq\gamma$, giving $\P\paren{\abs{\ell}\geq\tau}\leq\delta$.
    The complementary part $h=f-\ell$ satisfies $\widehat h(S)=0$ for all $1\leq\abs{S}\leq D$, i.e., is resilient off the constant (it retains the mean of $f$), and takes values in $[-1-\tau,1+\tau]$ on the event $\set{\abs{\ell}\leq\tau}$.
    Truncating $h$ to this event, zeroing out the Fourier coefficients of the truncation on levels $1\leq\abs{S}\leq D$ (which are $O\paren{\delta k^{2D+2}}$-small in $\ell_\infty$ norm, exactly as in \cite[Claim 3.4]{dachman2014approximate}, since $\widehat h(S)=0$ on these levels), and rescaling produces a function $g:\set{-1,1}^k\to[-1,1]$ that is $D$-resilient off the constant and satisfies $\E[fg]\geq1-O\paren{\tau+\delta k^{2D+2}}$, and so $\norm{f-g}_1=1-\E[fg]\leq O\paren{\tau+\delta k^{2D+2}}$ as desired.
    For the corresponding generalization of the hardness result \cite[Theorem 1.1]{dachman2014approximate}, one centers the witness, i.e., replaces $g$ by $h'=g-\widehat g(\emptyset)$; the copies of $h'$ over the $(n,k,D)$-design of \cite[Lemma 2.1]{dachman2014approximate} remain pairwise orthogonal exactly as in their proof.
    One then answers every statistical query $\phi$ with the mean-only reference value $\widehat g(\emptyset)\E[\phi]$, which is itself realized by labels independent of the input, and the standard fooling argument shows that any SQ algorithm of small complexity outputs, for most planted targets, a hypothesis whose correlation with the target is at most $\abs{\mu}+2\alpha+o(1)$, i.e., whose excess error on these (realizable) instances is at least $\frac{1-\abs{\mu}}{2}-\alpha-o(1)$.
    We leave the remaining details to the interested reader.
\end{remark}

\section{Non-separable functions}\label{sec:non_separable}
\subsection{Correlated probability spaces}\label{subsec:correlated_spaces}
If we hope to generalize \Cref{theorem:multilinear} beyond multilinear functions, we will need to strengthen our initial assumption, i.e., that $\Corr(X_i,Y_i)\leq\rho$.
Because $\Corr(X_i,Y_i)$ only captures linear relationships between $X_i$ and $Y_i$, if our function is not multilinear, it has the potential to have higher correlation $\Corr(f(X),f(Y))$ than that of its components, i.e., $\Corr(X_i,Y_i)$, by utilizing these nonlinear relationships.
Recall \Cref{example:cos_arccos,example:recursive_majority}, which illustrate how nonlinear relationships in a single variable can lead to noise stability that does not decay with depth.
To this end, we broaden our concept of correlation via the following definition of correlation between probability spaces, sometimes referred to as the ``Hirschfeld--Gebelein--R\'enyi maximal correlation'' as it was first introduced by Hirschfeld~\cite{hirschfeld1935connection} and Gebelein~\cite{Gebelein:41} and then studied by R\'enyi~\cite{Renyi:59}.
\begin{definition}
    Given a probability measure $\P$ defined on $\Omega=\prod_{i=1}^k \Omega_i$, we say that $\Omega_1,\dots,\Omega_k$ are \dfn{correlated spaces}.
    Given two linear subspaces $A$ and $B$ of $L^2(\P)$, we define the \dfn{correlation} between $A$ and $B$ by
    \[ \rho(A,B;\P)=\rho(A,B)=\sup\set{\Corr(f,g):f\in A, g\in B} \]
    For $i,j\in[k]$, we define the \dfn{correlation} $\rho(\Omega_i,\Omega_j;\P)$ by
    \[ \rho(\Omega_i,\Omega_j;\P)=\rho(\Omega_i,\Omega_j)=\rho(L^2(\Omega_i,\P),L^2(\Omega_j,\P);\P). \] 
    As $\rho$ also denotes a parameter in this paper, we will also use the notation $\Corr(\cdot)$ in place of $\rho(\cdot)$ for clarity. 
\end{definition}
For random vector $X:\Omega\to\R^n$, let $L^2(X)$ denote the linear subspace of $L^2(\P)$ that is $\sigma(X)$-measurable.
Each $f\in L^2(X)$ can be identified with a function $\mf f:\R^n\to\R$ satisfying $\E[\mf f(X)^2]<\infty$, and we will frequently make no distinction between these two functions.
Equipped with this new definition, we will consider the condition that $\Corr(L^2(X_i),L^2(Y_i))\leq\rho$, in place of the condition that $\abs{\Corr(X_i,Y_i)}\leq\rho$.

We recall some theory from Mossel~\cite[\S 2.2]{Mossel_Gaussian} concerning noise correlation of functions.
For two random variables $X,Y:\Omega\to\R$, if $f\in L^2(X)$ satisfies $\E[f(X)]=0$ and $\E[f(X)^2]=1$, then the maximizer of $\abs{\E[f(X)g(Y)]}$ among $g\in L^2(Y)$ satisfying $\E[g(Y)^2]=1$ is given by 
\[ g=\frac{\E[f(X)|Y]}{\norm{\E[f(X)|Y]}_2} = \frac{\E[f(X)|Y]}{\sqrt{\E[\E[f(X)|Y]^2]}}, \]
and
\[ \abs{\E[f(X)g(Y)]} = \norm{\E[f(X)|Y]}_2\]
The function $\E[f(X)|Y]\in L^2(Y)$ is the \emph{Markov operator} from $L^2(X)$ to $L^2(Y)$, and as such we will denote this $g$ by $Tf\in L^2(Y)$.
We have $\E[Tf]=0$, so
\begin{equation}\label{eq:correlation_using_Markov}
    \Corr(L^2(X),L^2(Y)) = \sup\set{\E[f(X)Tf(X)]: f\in L^2(X),\E[f]=0,\E[f^2]=1}.
\end{equation}
\subsection{Efron--Stein decomposition}\label{subsec:Efron_Stein}
Returning to our setup in \Cref{sec:multilinear}, where $\P$ is a product measure on product space $\Omega$ over the laws of correlated pairs of random variables $(X_1,Y_1),\dots,(X_n,Y_n)$, we now introduce the Efron--Stein decomposition as our analog to the $\chi_S$ that decomposed our multilinear polynomials.
\begin{definition}
    Let $(\Omega_1,\mu_1),\dots,(\Omega_n,\mu_n)$ be laws of correlated pairs of real-valued random variables $(X_1,Y_1),\dots,(X_n,Y_n)$, and let $(\Omega,\mu)=\prod_{i=1}^n(\Omega_i,\mu_i)$.
    The \dfn{Efron--Stein decomposition} of $f\in L^2(X)$ is given by
    \[ f(x) = \sum_{S\subseteq[n]}f_S(x), \]
    where each function $f_S$ depends only on $x_S$ and for all $S\not\subseteq S'$, we have $\E[f_S|X_{S'}]=0$ almost surely.
    We can similarly define the Efron--Stein decomposition for any $g\in L^2(Y)$.
\end{definition}
It is well-known that the Efron--Stein decomposition exists and is unique \cite{EfronStein:81}.
A standard result (see, for example, \cite[Proposition 2.11]{Mossel_Gaussian}, which only proves the result in the discrete case but naturally holds in general) states that the Efron--Stein decomposition ``commutes'' with the Markov operator, i.e., $(Tf)_S = T(f_S)\in L^2(Y)$, where $T:L^2(X)\to L^2(Y)$ is the Markov operator, which is given by $T=\bigotimes_{i=1}^n T_i$ for Markov operators $T_i:L^2(X_i)\to L^2(Y_i)$.
It thus is unambiguous to write $Tf_S$.
Similarly generalizing the proof of \cite[Proposition 2.12]{Mossel_Gaussian}, if $\Corr(L^2(X_i),L^2(Y_i))\leq\rho_i$ for some $\rho_i$ for each $i$, then for all $S\subseteq[n]$,
\begin{equation}\label{eq:Markov_operator_correlation_bound}
    \norm{Tf_S}_2 \leq \paren{\prod_{i\in S}\rho_i}\norm{f_S}_2.
\end{equation}
\subsection{Hierarchical non-separable functions}\label{subsec:hierarchical_nonseparable}
Recall the goal of this section is to generalize \Cref{theorem:multilinear} beyond multilinear functions.
Just as we had to broaden our assumption that $\Corr(X_i,Y_i)\leq\rho$ to $\Corr(L^2(X_i),L^2(Y_i))\leq\rho$, we must also broaden our assumption that $f$ is non-linear to account for non-linear relationships in a single variable.
To this end, we replace the concept of non-linearity with non-separability, where recall an (additively) separable function $f:\R^n\to\R$ is a function of the form $f(x)=f_1(x_1)+\cdots+f_n(x_n)$ for $f_1,\dots,f_n:\R\to\R$.
Note that in the context of multilinear polynomials, separability and linearity are equivalent, so this definition of a separable function is consistent with our setup in \Cref{sec:multilinear}.
We can then define non-separability by the condition that
\[ \Corr(L^2(f(X)),L^2(X_1)+\cdots+L^2(X_n))\leq 1-\varepsilon. \]
For any $g\in L^2(f(X))$ and $h_i\in L^2(X_i)$ for all $i\in[n]$, where we normalize $\E[g]=0=\E[h_i]$ for all $i$ and $\E[g^2]=1=\E[(h_1+\cdots+h_n)^2]$, we have by the defining properties of the Efron--Stein decomposition, where we view $G=g\circ f\in L^2(X)$,
\begin{align*}
    \Corr(g,h_1+\cdots+h_n)
    &= \E[g(f(X))(h_1(X_1)+\cdots+h_n(X_n))]
    \\ &= \sum_{S\neq\emptyset}\E[G_S(h_1+\cdots+h_n)_S]
    \\ &= \sum_{i=1}^n\E[G_{\{i\}} h_i]
    \\ &\leq \sum_{i=1}^n\norm{G_{\{i\}}}_2\norm{h_i}
    \\ &\leq \sqrt{\sum_{i=1}^n\norm{G_{\{i\}}}_2^2}\sqrt{\sum_{i=1}^n\norm{h_i}_2^2}
    \\ &= \sqrt{\sum_{i=1}^n\norm{G_{\{i\}}}_2^2}.
\end{align*}
Note that by picking $h_i$ appropriately, one can make both applications of the Cauchy--Schwarz inequality tight, and thus
\begin{equation}\label{eq:epsilon_equivalence_1}
    \Corr(L^2(f(X)),L^2(X_1)+\cdots+L^2(X_n)) = \sup_{\substack{g\in L^2(f(X))\\\E[g]=0,\E[g^2]=1}}\sqrt{\sum_{i=1}^n\norm{G_{\{i\}}}_2^2}.
\end{equation}
Because of the square root on the right hand side, we will reparametrize $\varepsilon$ so that our non-separability condition is
\[ \Corr(L^2(f(X)),L^2(X_1)+\cdots+L^2(X_n))\leq\sqrt{1-\varepsilon}. \]
By naturally extending our definition of $d(f,V)$ to $d(U,V)$ for $U,V\subseteq L^2(\Omega,\P)$ by
\[ d(U,V) = \inf_{f\in U,g\in V}\frac{\E[(f-g)^2]}{\Var(f)}, \]
note that
\begin{align}
    \notag
    d(L^2(f(X)),L^2(X_1)+\cdots+L^2(X_n))
    &= \inf_{g,h_1,\dots,h_n}\frac{\displaystyle\sum_{S}\norm{G_S-(h_1)_S-\cdots-(h_n)_S)}_2^2}{\displaystyle\sum_{S\neq\emptyset}\norm{G_S}_2^2}
    \\ &= \inf_{\substack{g\in L^2(f(X))\\\E[g]=0,\E[g^2]=1}}\displaystyle\sum_{\abs{S}\geq2}\norm{G_S}_2^2, \label{eq:epsilon_equivalence_2}
\end{align}
where we have used the fact that $f\mapsto f_S$ is a linear operator, which follows from linearity of expectation.
Assuming $\E[g]=0$ and $\E[g^2]=1$, we have $G_\emptyset\equiv0$ and $\sum_{S\subseteq[n]}\norm{G_S}_2^2=1$, so we see that the extremizers in \eqref{eq:epsilon_equivalence_1} and \eqref{eq:epsilon_equivalence_2} are the same $g$, and letting $V=L^2(X_1)+\cdots+L^2(X_n)$,
\begin{equation}\label{eq:epsilon_equivalence}
    \Corr(L^2(f(X)),V)\leq\sqrt{1-\varepsilon} \iff d(L^2(f(X)),V) \geq \varepsilon.
\end{equation}
While the latter condition is more consistent with our condition in \Cref{sec:multilinear}, we will use the former condition to consistently use the language of correlated probability spaces.
Of course, this choice is purely stylistic, as these two conditions are equivalent.
We say a function $f\in L^2(X)$ satisfying this condition is \dfn{$\varepsilon$-non-separable}.
Likewise, we can define a \dfn{hierarchical function} $f\in L^2(X)$ of non-separability $\varepsilon>0$ and depth $d$ analogously to how we defined hierarchical multilinear functions.

We are now ready to prove \Cref{lemma:nonseparable}, our generalization of \Cref{lemma:multilinear} to non-separable functions.
\begin{proof}[Proof of \Cref{lemma:nonseparable}]
    Let $Z=f(X)$ and $W=f(Y)$.
    Pick arbitrary $g,h:\R\to\R$ such that $\E[g(Z)]=\E[h(W)]=0$ and $\E[g(Z)^2]=\E[h(W)^2]=1$.
    It suffices to show
    \[ \E[g(Z)h(W)]\leq(1-\varepsilon)\rho+\varepsilon\rho^2. \]
    Let
    \[ G=g\circ f\in L^2(X) \quad\text{and}\quad H=h\circ f\in L^2(Y). \]
    Note that for distinct $S,S'\subseteq[n]$, we have $\E[G_S(X) H_{S'}(Y)]=0$.
    This is because we either have $S\not\subseteq S'$ or $S'\not\subseteq S$, and in the former case, by the defining properties of the Efron--Stein decomposition and recalling the $(X_i,Y_i)$ are mutually independent across $i$,
    \begin{align*}
        \E[G_S(X)H_{S'}(Y)]
        &= \E[\E[G_S(X)H_{S'}(Y)|(X_{S'},Y_{S'})]]
        \\ &= \E[\E[G_S(X)|(X_{S'},Y_{S'})]H_{S'}(Y)]
        \\ &= \E[0\cdot H_{S'}(Y)]
        \\ &= 0.
    \end{align*}
    Thus,
    \begin{align*}
        \E[g(Z)h(W)]
        &= \sum_{S\neq\emptyset}\E[G_S(X)H_S(Y)]
        \\ &= \sum_{S\neq\emptyset}\norm{G_S}_2\norm{H_S}_2\E\bracket{\frac{G_S(X)}{\norm{G_S}_2}\frac{H_S(Y)}{\norm{H_S}_2}}
        \\ &\leq \sum_{S\neq\emptyset}\norm{G_S}_2\norm{H_S}_2\E\bracket{\frac{G_S(X)}{\norm{G_S}_2}T\paren{\frac{G_S}{\norm{G_S}_2}}}
        \\ &= \sum_{S\neq\emptyset}\norm{G_S}_2\norm{H_S}_2\norm{TG_S}_2
        \\ &\leq \sum_{S\neq\emptyset}\norm{G_S}_2\norm{H_S}_2\rho^{\abs{S}},
    \end{align*}
    where the first inequality uses the property that the Markov operator maximizes covariance as discussed in \Cref{subsec:correlated_spaces}, and the second inequality follows from \eqref{eq:Markov_operator_correlation_bound}.
    Applying the Cauchy--Schwarz inequality and then our non-separability condition, we bound
    \begin{align*}
        \E[g(Z)h(W)]
        &\leq \sqrt{\sum_{S\neq\emptyset}\norm{G_S}_2^2\rho^{\abs{S}}}\sqrt{\sum_{S\neq\emptyset}\norm{H_S}_2^2\rho^{\abs{S}}}
        \\&\leq (1-\varepsilon)\rho+\varepsilon\rho^2,
    \end{align*}
    which completes the proof.
\end{proof}
\Cref{theorem:nonseparable} follows from the same proof as that of \Cref{theorem:multilinear} using \Cref{lemma:nonseparable} in place of \Cref{lemma:multilinear}.
\begin{remark}\label{remark:nonseparable_epsilon1}
    Again, we note that in the case $\varepsilon=1$, i.e., $f$ is a hierarchical function where each component function has zero correlation with any function of a single one of its component variables, \Cref{lemma:nonseparable} implies $\Corr(L^2(f(X)),L^2(f(Y)))\leq\rho^{2^d}$.
    So, analogously to \Cref{remark:multilinear_epsilon1}, in this case we have doubly exponential decay in depth with decay rate depending on $\rho$, stronger than the singly exponential decay provided by \Cref{theorem:nonseparable} with decay rate depending on $\varepsilon$.
\end{remark}
Next, we state and prove the following general version of \Cref{lemma:low_degree_multilinear}, which we use to justify \eqref{eq:low_degree_general}.
\begin{lemma}\label{lemma:low_degree_general}
    Suppose $\{(X_i,Y_i):i\in[n]\}$ are mutually independent, where for some $0\leq\rho\leq1$, we have $\Corr(L^2(X_i),L^2(Y_i))=\rho$ for all $i$.
    Suppose $X$ and $Y$ have the same marginal distributions.
    If $f:\R^n\to\R$ satisfies $\E[f(X)^2]<\infty$, then let $M$ denote the maximum correlation between $L^2(f)$ and any function of Efron--Stein degree at most $D$, i.e.,
    % \[ M = \Corr(L^2(f(X)),\mc G), \]
    $M = \Corr(L^2(f(X)),\mc G),$
    where
    % \[ \mc G = \sum_{\abs{S}\leq D} L^2(X_S). \]
    $\mc G = \sum_{\abs{S}\leq D} L^2(X_S).$
    Let
    \[ \gamma = \sup_{G,H}\sum_{S\neq\emptyset}\norm{G_S}_2\norm{H_S}_2\rho^{\abs{S}}, \]
    where the supremum is taken over all $g,h:\R\to\R$ such that $\E[g(f(X))]=\E[h(f(Y))]=0$ and $\E[g(f(X))^2]=\E[h(f(Y))^2]=1$, and $G=g\circ f \in L^2(X)$ and $H=h\circ f\in L^2(Y)$.
    Then
    \[ M \leq \rho^{-D/2}\sqrt{\gamma}. \]
\end{lemma}
\begin{proof}
    Let $Z=f(X)$ and $W=f(Y)$.
    Using a slightly generalized but essentially similar argument as that which we used to show \eqref{eq:epsilon_equivalence_1}, we have
    \[ M = \sup_{\substack{g\in L^2(f(X))\\\E[g]=0,\E[g^2]=1}}\sqrt{\sum_{1\leq\abs{S}\leq D}\norm{G_S}_2^2}, \]
    where $G=g\circ f\in L^2(X)$.
    It then suffices to show that for any such $g$,
    \[ \sqrt{\sum_{1\leq\abs{S}\leq D}\norm{G_S}_2^2} \leq \rho^{-D/2}\paren{\sup_{\substack{h\in L^2(f(Y))\\\E[h]=0,\E[h^2]=1}}\sum_{S\neq\emptyset}\norm{G_S}_2\norm{H_S}_2\rho^{\abs{S}}}^{1/2}, \]
    where $H=h\circ f \in L^2(Y)$.
    Letting $g$ and $h$ be the same when viewed as functions $\R\to\R$, because $X$ and $Y$ have the same marginal distributions, we note that $H_S=G_S$ for all $S$, which achieves the equality case of the Cauchy--Schwarz inequality that implies
    \[ \sup_{\substack{h\in L^2(f(Y))\\\E[h]=0,\E[h^2]=1}}\sum_{S\neq\emptyset}\norm{G_S}_2\norm{H_S}_2\rho^{\abs{S}} = \sum_{S\neq\emptyset}\norm{G_S}_2^2\rho^{\abs{S}}. \]
    Hence, it suffices to show
    \[ \sum_{1\leq\abs{S}\leq D}\norm{G_S}_2^2 \rho^{D} \leq \sum_{S\neq\emptyset}\norm{G_S}_2^2\rho^{\abs{S}}, \]
    which is clearly true as $0 \leq \rho \leq 1$.
\end{proof}
To justify \eqref{eq:low_degree_general}, note that for any $0 \leq \rho < 1$, by \Cref{remark:correlation_construction_L2}, we can construct $Y$ satisfying the conditions of \Cref{lemma:low_degree_general}.
Recall that \Cref{theorem:nonseparable} was proven by inductively applying \Cref{lemma:nonseparable}.
Using the fact that the proof of \Cref{lemma:nonseparable} implies the stronger technical statement
\[ \sup_{G,H} \sum_{S\neq\emptyset}\norm{G_S}_2\norm{H_S}_2\rho^{\abs{S}} \leq (1-\varepsilon)\rho+\varepsilon\rho^2, \]
the proof of \Cref{theorem:nonseparable} implies that for a hierarchical function $f\in L^2(X)$ of non-separability $\varepsilon>0$ and depth $d$, and for any $\delta<\varepsilon$ and $0 \leq \rho < 1$ (where because our constructed $Y$ has the same marginal distribution as $X$, we have $f$ is also such a hierarchical function in $L^2(Y)$),
\[ \sup_{G,H} \sum_{S\neq\emptyset}\norm{G_S}_2\norm{H_S}_2\rho^{\abs{S}} \leq (1-\delta)^{d-C_{\varepsilon,\delta}\log\paren{\frac{1}{1-\rho}}}. \]
Thus, \Cref{lemma:low_degree_general} implies \eqref{eq:low_degree_general}.

Lastly, we use \Cref{theorem:nonseparable} to prove \Cref{theorem:sq_learning_Gaussian}.
\begin{proof}[Proof of \Cref{theorem:sq_learning_Gaussian}]
    By \cite[Theorem 1.7]{diakonikolas2021optimality}, it suffices to show that for any polynomial $p:\R^m\to\R$ of degree less than $D=n^c$, we have
    \[ \E[\abs{f(X)-p(X)}]\geq\frac{1-\E[f(X)]^2}{8}-o(1). \]
    To show this we will use \cite[Theorem 1.5]{diakonikolas2021optimality}.
    By \Cref{theorem:nonseparable},
    \[ \Corr(L^2(f(X)),L^2(f(Y))) \leq (1-\delta)^{d-\Theta(\log\frac{1}{1-\rho})}. \]
    Fixing $\delta<\varepsilon$ and setting $1-\rho=\frac{1}{D^3}$, this bound is $e^{-\Theta(\log n)+\Theta(\log D)}$, and choosing $c$ sufficiently small ensures this is $o(1)$.
    For any such polynomial $p$, by \cite[Theorem 1.5]{diakonikolas2021optimality}, we have
    \[ \E[\abs{f(X)-p(X)}]\geq\frac{\P(f(X)\neq f(Y))}{4}-O(D\sqrt{1-\rho}) = \frac{\frac{1}{2}-\frac{1}{2}\E[f(X)f(Y)]}{4} - o(1), \]
    where $Y$ is a $\rho$-correlated spherical Gaussian.
    By \eqref{eq:correlation_using_Markov}, $\Corr(L^2(X_i),L^2(Y_i))=\rho$ for all $i$.
    Moreover, as $X$ and $Y$ have the same marginal distribution, we have $\Var(f(X))=\Var(f(Y))=1-\E[f(X)]^2$, and
    \[ \E[f(X)f(Y)] = \E[f(X)]^2+\Cov(f(X),f(Y)) \leq \E[f(X)]^2+\Var(f(X))\Corr(L^2(f(X)),L^2(f(Y))). \]
    Indeed, if $\Var(f(X))=0$ this holds with equality, as both sides equal $\E[f(X)]^2$; otherwise, the inequality holds because $\frac{f(X)-\E[f(X)]}{\sqrt{\Var(f(X))}}$ and $\frac{f(Y)-\E[f(Y)]}{\sqrt{\Var(f(Y))}}$ are mean-zero, unit-variance elements of $L^2(f(X))$ and $L^2(f(Y))$, respectively, so their correlation is at most $\Corr(L^2(f(X)),L^2(f(Y)))$.
    Hence
    \begin{align*}
        \E[\abs{f(X)-p(X)}]
        &\geq \frac{1}{8}-\frac{\E[f(X)f(Y)]}{8} - o(1)
        \\ &\geq \frac{1-\E[f(X)]^2}{8}-\frac{\Var(f(X))\Corr(L^2(f(X)),L^2(f(Y)))}{8}-o(1)
        \\ &\geq \frac{1-\E[f(X)]^2}{8}-\frac{(1-\delta)^{d-\Theta(\log\frac{1}{1-\rho})}}{8}-o(1)
        \\ &= \frac{1-\E[f(X)]^2}{8}-o(1),
    \end{align*}
    where the last inequality uses $\Var(f(X))\leq1$, and the final equality holds as $(1-\delta)^{d-\Theta(\log\frac{1}{1-\rho})}=o(1)$, as shown above.
    This completes the proof.
\end{proof}

\section*{Acknowledgements}
R.L.\ was partially supported by a Hertz Fellowship and a PD Soros Fellowship.
E.M.\ was partially supported by NSF DMS-2031883, Bush Faculty Fellowship ONR-N00014-20-1-2826, and Simons Investigator award (622132).

\bibliographystyle{amsinit}
\bibliography{ref,all,my}

\end{document}